\documentclass[a4paper,10pt]{article}
\usepackage{geometry}
\usepackage{amsthm}
\usepackage{amssymb}
\usepackage{amsmath}
\usepackage{xcolor}
\newtheorem{thm}{Theorem}[section]

\newtheorem{lemma}[thm]{Lemma}

\usepackage{cite}
\usepackage[colorlinks=true,anchorcolor=blue,filecolor=blue,linkcolor=red,
urlcolor=blue,citecolor=blue,hypertexnames=false]{hyperref}

\newcommand{\F}{\mathcal F}
\newcommand{\Hfun}{\mathcal H}

\begin{document}

\title{An exact formula for Erd\H{o}s' problem 1005}
\author{
Yanmohan Wang\thanks{Institute for Interdisciplinary Information Sciences, Tsinghua University, Beijing 100084, China. Email: {\tt wymh25@mails.tsinghua.edu.cn}}
\and 
Mingxu Xie\thanks{Qiuzhen College, Tsinghua University, Beijing 100084, China.
Email: {\tt xiemx24@mails.tsinghua.edu.cn}}
\and 
Ziyuan Zhao\thanks{School of Mathematical Sciences, University of Science and Technology of China, Hefei, Anhui 230026, China. Research supported by Innovation Program for Quantum Science and Technology 2021ZD0302902. Email: {\tt zyzhao2024@mail.ustc.edu.cn}}}
\date{}
\maketitle

\begin{abstract}
In 1943, Erd\H{o}s considered the minimum number $f(n)$ of terms between two fractions in the Farey sequence of order $n$ whose numerators and denominators are oppositely ordered. Determining the constant $c$ in $f(n)=(c+o(1))n$ is known as Erd\H{o}s Problem 1005. Recently, Cipollini solved this asymptotic problem by proving that $f(n)=(1/4+o(1))n$. Following his framework, we give an analytic proof of an exact formula for $f(n)$ for all sufficiently large $n$. Combining this with a finite computer verification, we further determine $f(n)$ for every integer $n\geq 4$.
\end{abstract}

\section{Introduction}

For a positive integer $n$, let the Farey sequence of order $n$ be
$\F_n=(a_0/b_0,a_1/b_1,\cdots)$, the increasing sequence of all reduced
fractions $a/b\in[0,1]$ with $b\le n$.
For $n\ge4$, define
\begin{equation}\label{eq:f-definition}
 f(n):=\min\left\{j-i-1:
 0\le i<j<|\F_n|, 
 (a_j-a_i)(b_j-b_i)<0\right\}.
\end{equation}

In 1942, Mayer initiated the study of the function
$f(n)$ and proved that $f(n)\to\infty$ \cite{Mayer1,Mayer2}.  In 1943,
Erd\H{o}s proved the first linear lower bound,
$f(n)\ge n/400-O(1)$~\cite{Erdos}.  The related problem of determining
whether there exists a constant $c>0$ satisfying $f(n)=(c+o(1))n$ and,
if so, determining its value is Erd\H{o}s Problem 1005
\cite{Bloom}.  In 2025, van Doorn improved the lower
bound to $f(n)\ge(1/12-o(1))n$ \cite{vanDoorn}.  He also gave explicit
constructions yielding the following four upper bounds.
For every integer $n\ge4$, write $n=4m+r$, where $m\ge1$ and
$r\in\{0,1,2,3\}$, and define
\[
 U(n):=
 \begin{cases}
  m+1,&r=0,\\
  m+2,&r\in\{1,2\},\\
  m+4,&r=3.
 \end{cases}
\]

\begin{thm}[van Doorn~{\cite[Theorem~1]{vanDoorn}}]\label{thm:vandoorn-upper}
For every integer $n\geq 4$, $f(n)\leq U(n)$.
\end{thm}

In the same paper, van Doorn computed $f(n)$ for $n\le5000$. Equality in
Theorem~\ref{thm:vandoorn-upper} holds for every
$92\le n\le5000$, whereas $91$ is the largest integer $n<5000$ for which
equality fails.  He therefore conjectured that equality holds for every
$n\ge92$.  Very recently, Cipollini proved $f(n)=(1/4+o(1))n$, thereby
determining the asymptotic constant in Erd\H{o}s Problem 1005
\cite{Cipollini}. In this paper, we determine $f(n)$ exactly for every integer
$n\ge4$. In particular, van Doorn's conjecture
hold for all sufficiently large $n$.

\begin{thm}\label{thm:main}
For all sufficiently large integers $n$, $f(n)=U(n)$.
\end{thm}

The proof of Theorem~\ref{thm:main} follows the analytic framework developed
in~\cite{Cipollini}.  The following theorem, proved with computer assistance,
determines $f(n)$ exactly for every integer $n\ge4$.

\begin{thm}\label{thm:all-n}
Let $\mathcal E:=\{7,9,11,15,19,23,25,27,31,35,39,49,51,63,91\}.$
Then, for every integer $n\ge4$,
\[
 f(n)=
 \begin{cases}
  U(n)-2,&n\in\{15,27\},\\
  U(n)-1,&n\in\mathcal E\setminus\{15,27\},\\
  U(n),&n\notin\mathcal E.
 \end{cases}
\]
\end{thm}

The main focus of this paper is the analytical proof of
Theorem~\ref{thm:main}.  A proof sketch of Theorem~\ref{thm:all-n}, together
with the code for its verification, is given in
Section~\ref{sec:5}.

\section{Preliminaries}\label{sec:2}
In this section, we introduce several auxiliary functions and establish their basic properties for later use.
We use the following notation throughout the paper.  For a positive integer
$m$, let $\varphi(m):=|\{1\le r\le m:\gcd(r,m)=1\}|$ be Euler's totient
function, and let $\Phi(m):=\sum_{r=1}^m\varphi(r)$. Thus
$\Phi(m)=|\F_m\setminus \{0\}|$.

\begin{lemma}\label{lem:totient}
For every integer $m\ge1$,
\begin{equation}\label{eq:totient}
 \Phi(m)\ge\frac27m(m+1).
\end{equation}
\end{lemma}

\begin{proof}
Writing $\mu$ for the M\"obius function, M\"obius inversion gives
\[
 \varphi(r)=r\sum_{d\mid r}\frac{\mu(d)}d.
\]
Hence
\[
\begin{aligned}
 \Phi(m)
 &=\sum_{r\le m}\sum_{d\mid r}\mu(d)\frac rd
 =\sum_{d\le m}\mu(d)
   \sum_{k\le \lfloor\frac md\rfloor}k \\
 &=\frac12\sum_{d\le m}\mu(d)
 \left\lfloor\frac md\right\rfloor
 \left(\left\lfloor\frac md\right\rfloor+1\right).
\end{aligned}
\]
Since $x^2-x\leq \lfloor x\rfloor(\lfloor x\rfloor+1)\leq x^2+x$ holds
for every $x\geq1$, the estimates
\[
 \sum_{d\ge1}\frac{\mu(d)}{d^2}=\frac6{\pi^2},\qquad
 \sum_{d>m}\frac1{d^2}\le\frac1m,\qquad
 \sum_{d\le m}\frac1d\le\log m+1,
\]
give
\begin{align*}
 \Phi(m)&=\frac12\sum_{d\le m}\mu(d)
 \left\lfloor\frac md\right\rfloor
 \left(\left\lfloor\frac md\right\rfloor+1\right)\\
 &\geq \frac12\sum_{d\le m}\mu(d)
 \left(\frac md\right)^2-\frac12\sum_{d\le m}|\mu(d)|
 \frac md\\
 &\geq 3m^2/\pi^2-(m/2)\log m-m.
\end{align*}
For $m\ge360$, the inequalities $\pi<22/7$ and $\log m\le m/60$ give
$\Phi(m)-2m(m+1)/7\ge983m^2/101640-9m/7>0$.
For $m<360$, the result follows from the exact integer computation described
in Section~\ref{sec:5}.
This proves \eqref{eq:totient}.
\end{proof}

The following auxiliary function $G$ is also used in~\cite[Section~3]{Cipollini}:
\[
 G(x):=\sum_{1\le r<x}\frac{\varphi(r)}r\left(1-\frac rx\right),
 \qquad G(0):=0.
\]

\begin{lemma}\label{lem:G}
The function $G$ has the following properties.
\begin{enumerate}
\item[(i)] $G(x+1)-G(x)\ge5/18$ for every $x\ge1$.

\item[(ii)] $G(x)\ge5x/18$ for every $x\ge3$, and $G(x)\ge x/4$ for
every $x\ge2$.

\item[(iii)] The function $G$ is increasing, and for every $0\le x<y$,
\[
 \frac5{18}(y-x-1)\le G(y)-G(x)\le y-x.
\]
\end{enumerate}
\end{lemma}

\begin{proof}
For every $x\ge1$,
\[
 G(x)=\sum_{r=1}^{\lfloor x\rfloor}\frac{\varphi(r)}r
      -\frac{\sum_{r=1}^{\lfloor x\rfloor}\varphi(r)}x
     =\sum_{r=1}^{\lfloor x\rfloor}\frac{\varphi(r)}r
      -\frac{\Phi(\lfloor x\rfloor)}x.
\]
It follows that, for every positive integer $m$,
\begin{equation}\label{eq:G-identities}
 G(m+1)-G(m)=\frac{\Phi(m)}{m(m+1)},\qquad
 G'(x)=\frac{\Phi(m)}{x^2}\quad(m<x<m+1).
\end{equation}

We first prove (i).  If $1\le x\le2$, then
\[
 G(x+1)-G(x)=\frac1x+\frac12-\frac2{x+1}\ge\frac13>\frac5{18}.
\]
If $2\le x\le3$, then
\begin{align*}
 G(x+1)-G(x)
 &=\frac2x+\frac23-\frac4{x+1}\ge4\sqrt2-\frac{16}{3}>\frac5{18}.
\end{align*}
For $x\ge3$, integrating the derivative in \eqref{eq:G-identities} on
$[x,\lceil x\rceil]$ and $[\lceil x\rceil,x+1]$ gives
\begin{align*}
 G(x+1)-G(x)
 &=\Phi(\lfloor x\rfloor)
   \left(\frac1x-\frac1{\lceil x\rceil}\right)
  +\Phi(\lceil x\rceil)
   \left(\frac1{\lceil x\rceil}-\frac1{x+1}\right)\\
 &\ge\frac27\left(
  1-\frac{(x-\lfloor x\rfloor)(\lceil x\rceil-x)}{x(x+1)}
  \right)\ge\frac{47}{168}>\frac5{18}.
\end{align*}
Here the first inequality follows from Lemma~\ref{lem:totient}; the second uses
$(x-\lfloor x\rfloor)(\lceil x\rceil-x)\le1/4$ and $x(x+1)\ge12$.

For (ii), Lemma~\ref{lem:totient} and \eqref{eq:G-identities} give
$G(m+1)-G(m)\ge2/7$ for every positive integer $m$.  Since $G(3)=5/6$
and $2/7>5/18$, we have $G(m)\ge5m/18$ at every integer $m\ge3$.
Moreover, $G''(x)=-2\Phi(m)/x^3\le0$ on $(m,m+1)$.  Thus
$G(x)-5x/18$ is concave on $[m,m+1]$ and nonnegative at both endpoints.
Starting instead from $G(2)=1/2$ and using $2/7>1/4$, we have
$G(x)\ge x/4$ for $x\ge2$.  This proves~(ii).

For (iii), the derivative in~\eqref{eq:G-identities} is nonnegative and at most $1$, because
$\Phi(m)\le m(m+1)/2\le x^2$ whenever $x\in(m,m+1)$. Hence $G$ is increasing and
$G(y)-G(x)\le y-x$ whenever $0\le x<y$.
It remains to prove the lower bound in (iii).  If $y-x\le1$, it follows from
the monotonicity of $G$.  Suppose that $y-x>1$. If $x\ge1$, let
$k=\lfloor y-x\rfloor$. Item (i) and monotonicity of $G$ give
\[
 G(y)-G(x)\ge G(x+k)-G(x)\ge\frac5{18}k
 \ge\frac5{18}(y-x-1).
\]
When $0\le x<1$, the definition gives $G(x)=0$.  In this case,
if $1<y<2$, then
\[
 G(y)=\frac{y-1}{y}\ge\frac5{18}(y-1)
      \ge\frac5{18}(y-x-1).
\]
If $2\le y<3$, then item (ii) gives
\[
 G(y)\ge\frac y4\ge\frac5{18}(y-1)
      \ge\frac5{18}(y-x-1).
\]
If $y\ge3$, then
\[
 G(y)\ge\frac5{18}y\ge\frac5{18}(y-x-1).
\]
This completes the proof of Lemma~\ref{lem:G}.
\end{proof}

For $t>0$ and real numbers $X<Y$, define
\begin{equation}\label{eq:H-def}
 \Hfun_t(X,Y):=
 \sum_{h\in\mathbb Z\setminus\{0\}}
 \frac{\varphi(|h|)}{|h|}
 \operatorname{length}\{0<u\le t:Xu<h<Yu\}.
\end{equation}
The sum is finite, since the set in \eqref{eq:H-def} is empty whenever
$|h|\ge t\max\{|X|,|Y|\}$.

\begin{lemma}\label{lem:H-formula}
For $t>0$ and real numbers $X<Y$,
\begin{equation}\label{eq:H-formula}
 \Hfun_t(X,Y)=
 \begin{cases}
  t\bigl(G(tY)-G(tX)\bigr),&0\le X<Y,\\[1mm]
  t\bigl(G(-tX)-G(-tY)\bigr),&X<Y\le0,\\[1mm]
  t\bigl(G(-tX)+G(tY)\bigr),&X<0<Y.
 \end{cases}
\end{equation}
\end{lemma}

\begin{proof}
Suppose first that $0\le X<Y$.  For every $Z>0$ and every positive
integer $h$,
\[
 \operatorname{length}\{0<u\le t:h<Zu\}
 =
 \begin{cases}
  t-h/Z,&0<h<tZ,\\
 0,&h\ge tZ.
 \end{cases}
\]
Hence
\[
 \sum_{h=1}^{\infty}\frac{\varphi(h)}h
 \operatorname{length}\{0<u\le t:h<Zu\}
 =\sum_{1\le h<tZ}\frac{\varphi(h)}h
 \left(t-\frac hZ\right)
 =tG(tZ).
\]
This identity also holds for $Z=0$, since both sides are zero. We thus have
\begin{align*}
 \Hfun_t(X,Y)
 &=\sum_{h=1}^{\infty}\frac{\varphi(h)}h
 \operatorname{length}\{0<u\le t:h<Yu\}\\
 &\quad-\sum_{h=1}^{\infty}\frac{\varphi(h)}h
 \operatorname{length}\{0<u\le t:h<Xu\}\\
 &=tG(tY)-tG(tX).
\end{align*}
The cases $X<Y\le0$ and $X<0<Y$ follow similarly.
\end{proof}

We finish the section with two classical results.

\begin{thm}[Dirichlet's approximation theorem]\label{thm:dirichlet}
For every $\alpha\in[0,1]$ and every positive integer $m$, there is a
fraction $p/q\in\F_m$ such that $|q\alpha-p|\le1/m$.
\end{thm}

\begin{thm}[Dress~\cite{Dress}]\label{thm:dress}
For every positive integer $n$ and every $0\le\alpha\le1$,
\[
 \Phi(n)\left(\alpha-\frac1n\right)
 \le \left|\F_n\cap(0,\alpha]\right|
 \le \Phi(n)\left(\alpha+\frac1n\right).
\]
\end{thm}

\section{Farey subsequences not containing $1/2$}\label{sec:3}
This section focuses on the analysis of the interval
$(a/b,(a+1)/(b-1))$ when its left endpoint $a/b$ satisfies
$b-2a\notin\{1,2\}$. The following theorem is the formal statement in this
case.

\begin{thm}\label{thm:general}
Let $n\ge4$, and let $a/b\in\F_n$ satisfy $1\le a\le b-2$.
If $b-2a\notin\{1,2\}$, then
\[
 \left|\F_n\cap
 \left(\frac ab,\frac{a+1}{b-1}\right)\right|
 \ge\frac{5n}{18}-O(\sqrt n\log n).
\]
\end{thm}

The following two technical lemmas make the further assumption that $b>n/4$.
Throughout these two lemmas, we assume
\begin{equation}\label{eq:conditions}
\begin{gathered}
 n\ge4,\quad a/b\in\F_n,\quad 1\le a\le b-2,\quad
 b>n/4,\quad b-2a\notin\{1,2\},\\
 p/q\in\F_{\lfloor\sqrt n\rfloor},\quad
 |qa/b-p|\le1/\lfloor\sqrt n\rfloor,\quad A:=aq-bp.
\end{gathered}
\end{equation}
The existence of the required fraction $p/q$ follows from
Theorem~\ref{thm:dirichlet}.

\begin{lemma}\label{lem:determinant-count}
Under the conditions~\eqref{eq:conditions},
\[
 \left|\F_n\cap
 \left(\frac ab,\frac{a+1}{b-1}\right)\right|
 =
 \frac bq\Hfun_{n/b}(A,A+p+q)
 +O(\sqrt n\log n).
\]
\end{lemma}

\begin{proof}
Fix an integer $h$ and let
\[
 N_h:=\left\{(x,y)\in\mathbb Z^2:
  0<x<y\le n,\quad \gcd(x,y)=1, \quad
\frac ab<\frac xy<\frac{a+1}{b-1},
  \quad qx-py=h\right\}.
\]
Then we have
\begin{equation}\label{eq:sum-Nh}
 \left|\F_n\cap
 \left(\frac ab,\frac{a+1}{b-1}\right)\right|
 =\sum_{h\in\mathbb Z}|N_h|.
\end{equation}

Since $\gcd(p,q)=1$, we may choose integers $r$ and $s$ such that
$qr-ps=1$. Thus
\begin{equation*}
 \begin{pmatrix}x\\y\end{pmatrix}
 =\begin{pmatrix}r&p\\s&q\end{pmatrix}
 \begin{pmatrix}h\\j\end{pmatrix},
 \qquad
 \begin{pmatrix}h\\j\end{pmatrix}
 =\begin{pmatrix}q&-p\\-s&r\end{pmatrix}
 \begin{pmatrix}x\\y\end{pmatrix},
\end{equation*}
where the two displayed matrices are inverse integer matrices. Hence this gives
a bijection between the integer pairs $(x,y)$ satisfying $qx-py=h$ and
the integers $j$, and it also implies
\begin{equation}\label{eq:gcd-change}
 \gcd(x,y)=\gcd(h,j).
\end{equation}
For $y>0$, a direct calculation gives
\begin{equation}\label{eq:determinant-condition}
 \frac ab<\frac xy<\frac{a+1}{b-1}
 \quad\Longleftrightarrow\quad
 A\cdot\frac {sh+qj}{b}<h<\frac{b(A+p+q)}{b-1}\cdot\frac {sh+qj}{b}.
\end{equation}
Since $a/b>0$ and $(a+1)/(b-1)\le1$, the inequalities in
\eqref{eq:determinant-condition} also imply $0<x<y$.  Therefore
\eqref{eq:gcd-change} and \eqref{eq:determinant-condition} give
\begin{equation}\label{eq:Nh-coprime-j}
 |N_h|=\left|\{j\in I_h\cap\mathbb Z:\gcd(j,h)=1\}\right|,
\end{equation}
where $I_h$ is the following possibly empty interval:
\[
 I_h:=\left\{j\in\mathbb R:0<sh+qj\le n,\quad
 A\cdot\frac{sh+qj}{b}<h<
 \frac{b(A+p+q)}{b-1}\cdot\frac{sh+qj}{b}\right\}.
\]

For $h\ne0$, the change of variables $u=(sh+qj)/b$ gives
\begin{equation}
 \operatorname{length}(I_h)=\frac bq
 \operatorname{length}\left\{0<u\le\frac nb:
 Au<h<\frac{b(A+p+q)}{b-1}u\right\}\notag.
\end{equation}
Consequently, the definition of $\Hfun_t(X,Y)$ gives
\begin{equation}\label{eq:H-Ih-relation}
 \sum_{h\in\mathbb Z\setminus\{0\}}
 \frac{\varphi(|h|)}{|h|}\operatorname{length}(I_h)=\frac bq\Hfun_{n/b}\left(A,\frac{b(A+p+q)}{b-1}\right).
\end{equation}

When $h=0$, we have $(x,y)=j(p,q)$, and hence $j=1$ whenever
$(x,y)\in N_0$. Thus $|N_0|\le1$. For $h\ne0$ and every positive
divisor $d$ of $|h|$, the number of multiples of $d$ in $I_h$ is
$\operatorname{length}(I_h)/d+O(1)$. Let $\mu$ be the M\"obius
function. Applying M\"obius inversion to \eqref{eq:Nh-coprime-j}, we obtain
\[
 |N_h|
 =\sum_{d\mid |h|}\mu(d)
 \left(\frac{\operatorname{length}(I_h)}d+O(1)\right)
 =
 \frac{\varphi(|h|)}{|h|}\operatorname{length}(I_h)
 +O\left(\sum_{d\mid |h|}1\right)
 \qquad(h\ne0).
\]
If $I_h$ is empty, both $|N_h|$ and the main term in this estimate
vanish. Hence, summing over $h\ne0$ and using \eqref{eq:sum-Nh},
\eqref{eq:H-Ih-relation}, we obtain
\begin{align}\label{eq:fixed-h-count}
 \left|\F_n\cap
 \left(\frac ab,\frac{a+1}{b-1}\right)\right|=\frac bq\Hfun_{n/b}\left(
 A,\frac{b(A+p+q)}{b-1}\right)+O\left(
 \sum_{\substack{h\ne0\\I_h\ne\emptyset}}
 \sum_{d\mid |h|}1\right).
\end{align}

By the assumptions in~\eqref{eq:conditions},
\[
 |A|\le\frac b{\lfloor\sqrt n\rfloor}=O(\sqrt n),
 \qquad p+q\le2\lfloor\sqrt n\rfloor,
 \qquad |A+p+q|\le
 \left|\frac{b(A+p+q)}{b-1}\right|=O(\sqrt n).
\]
Moreover, if $I_h$ is nonempty, then
\[
 |h|\le K:=\frac nb\max\left\{|A|,
 \left|\frac{b(A+p+q)}{b-1}\right|\right\}=O(\sqrt n).
\]
Therefore
\[
 1+\sum_{\substack{h\ne0\\ I_h\ne\emptyset}}
 \sum_{d\mid |h|}1
 \le 1+2\sum_{d\le K}\left\lfloor\frac Kd\right\rfloor
 =O\bigl(1+K\log(2K)\bigr)
 =O\bigl(\sqrt n\log n\bigr).
\]
Substituting this into~\eqref{eq:fixed-h-count} gives
\begin{equation}\label{eq:direct-H-count}
 \left|\F_n\cap
 \left(\frac ab,\frac{a+1}{b-1}\right)\right|
 =\frac bq\Hfun_{n/b}\left(A,\frac{b(A+p+q)}{b-1}\right)
 +O(\sqrt n\log n).
\end{equation}

Finally, Lemmas~\ref{lem:H-formula}
and~\ref{lem:G}(iii) give
\[
 \left|\Hfun_{n/b}\left(A,\frac{b(A+p+q)}{b-1}\right)
 -\Hfun_{n/b}(A,A+p+q)\right|
 \le\left(\frac nb\right)^2\frac{|A+p+q|}{b-1}.
\]
Since $n/b<4$, $b/(b-1)\le3/2$, it follows that
\[
 \frac bq\left|\Hfun_{n/b}\left(A,\frac{b(A+p+q)}{b-1}\right)
 -\Hfun_{n/b}(A,A+p+q)\right|
 \le24|A+p+q|=O(\sqrt n).
\]
Together with~\eqref{eq:direct-H-count}, this proves Lemma~\ref{lem:determinant-count}.
\end{proof}

\begin{lemma}\label{lem:H-lower}
Under the conditions~\eqref{eq:conditions},
\[
 \Hfun_{n/b}(A,A+p+q)\ge\frac{5qn}{18b}.
\]
\end{lemma}

\begin{proof}
Let $t=n/b$. Since $a/b\in\F_n$, we have $t\ge1$.
Moreover, $0\le p\le q$ and $\gcd(p,q)=1$.

Suppose first that $0<p<q$.  If $A\ge0$ or
$A+p+q\le0$, then Lemmas~\ref{lem:H-formula} and~\ref{lem:G}(iii) give
\[
 \frac1t\Hfun_t(A,A+p+q)
 \ge\frac5{18}\bigl((p+q)t-1\bigr)
 \ge\frac5{18}q.
\]

Suppose that $A<0<A+p+q$.  The positive integers $-A$ and $A+p+q$ have
sum $p+q$, and
\[
 \frac1t\Hfun_t(A,A+p+q)
 =G(-At)+G((A+p+q)t).
\]
If $q=2$, then $p=1$. Since $A=2a-b$ and
$b-2a\notin\{1,2\}$, neither $A=-1$ nor $A=-2$ can occur. On the other
hand, the two positive integers have sum $3$, so one of these two values of
$A$ must occur, a contradiction. Thus $q\ge3$.

If one of $-A$ and $A+p+q$ equals $1$, the other is at
least $q$.  Hence Lemma~\ref{lem:G}(ii) gives
\[
 G(-At)+G((A+p+q)t)\ge\frac5{18}qt.
\]
If one equals $2$ and the other is at least $3$, then
Lemma~\ref{lem:G}(ii) gives
\[
 G(-At)+G((A+p+q)t)
 \ge\frac t2+\frac5{18}(p+q-2)t
 \ge\frac5{18}qt.
\]
If both integers equal $2$, then $(p,q)=(1,3)$ and
$2G(2t)\ge t>5qt/18$.  If both integers are at least $3$, then
\[
 G(-At)+G((A+p+q)t)
 \ge\frac5{18}(p+q)t
 \ge\frac5{18}qt.
\]

It remains to consider $p=0$ and $p=q$.  If $p=0$, then $q=1$ and
$A=a\ge1$.  By Lemma~\ref{lem:G}(i),
\[
 \frac1t\Hfun_t(A,A+1)
 =G((A+1)t)-G(At)
 \ge G(At+1)-G(At)\ge\frac5{18}.
\]
If $p=q$, then $p=q=1$ and $A=a-b\le-2$.  By
Lemma~\ref{lem:G}(iii),
\[
 \frac1t\Hfun_t(A,A+2)
 =G(-At)-G((-A-2)t)
 \ge\frac5{18}(2t-1)\ge\frac5{18}.
\]
This completes the proof.
\end{proof}

\begin{proof}[\textbf{\textup{Proof of Theorem~\ref{thm:general}}}]
Suppose first that $b\le n/4$.  Since
\[
 \frac{a+1}{b-1}-\frac ab=\frac{a+b}{b(b-1)}>\frac4n,
\]
the two estimates in Theorem~\ref{thm:dress} give
\[
 \left|\F_n\cap\left(\frac ab,\frac{a+1}{b-1}\right)\right|
 \ge\left|\F_n\cap\left(0,\frac{a+1}{b-1}\right]\right|
  -\left|\F_n\cap\left(0,\frac ab\right]\right|-1
 >\frac{2\Phi(n)}n-1.
\]
By Lemma~\ref{lem:totient},
$\Phi(n)>2n^2/7$, and hence the last quantity is greater than $4n/7-1$,
which suffices.

Now suppose that $b>n/4$, and let $p/q$ and $A$ be as in
\eqref{eq:conditions}. Lemma~\ref{lem:H-lower} gives
\[
 \frac bq\Hfun_{n/b}(A,A+p+q)
 \ge\frac bq\cdot\frac{5qn}{18b}=\frac{5n}{18}.
\]
Combining this inequality with Lemma~\ref{lem:determinant-count} proves the
theorem.
\end{proof}

\section{The remaining cases}\label{sec:4}

\begin{thm}\label{thm:central-cases}
Let $m\ge3$, $r\in\{0,1,2,3\}$, and $n=4m+r$.  Let $a/b\in\F_n$
satisfy $1\le a\le b-2$ and $b-2a\in\{1,2\}$.  If $r=3$, suppose in
addition that $a/b\ne(2m+1)/(4m+3)$.  Then
\[
 \left|\F_n\cap\left(\frac ab,\frac{a+1}{b-1}\right)\right|
 \ge U(n).
 \]
\end{thm}

\begin{proof}
Suppose first that $b-2a=2$.  Since $a/b$ is reduced, $a$ is odd, so
\[
 \frac ab=\frac{2t-1}{4t}
 \quad\text{and}\quad
 \left(\frac ab,\frac{a+1}{b-1}\right)
 =I_t:=\left(\frac{2t-1}{4t},\frac{2t}{4t-1}\right)
\]
for some $1\le t\le m$. 
Then $I_m\subseteq I_{m-1}\subseteq\cdots\subseteq I_1$.
For every positive integer $s$, $s/(2s+1)>(2m-1)/(4m)$ if and only if $s\geq m$.
Consider the following table:
\[
{\renewcommand{\arraystretch}{1.15}
\begin{array}{c|l|c}
r&\text{fractions}&\text{cardinality}\\[1mm]\hline
0&\left\{\dfrac{s}{2s+1}:m\le s\le2m-1\right\}
 \cup\left\{\dfrac12\right\}&m+1\\[4mm]
1,2&\left\{\dfrac{s}{2s+1}:m\le s\le2m\right\}
 \cup\left\{\dfrac12\right\}&m+2\\[4mm]
3&\left\{\dfrac{s}{2s+1}:m\le s\le2m+1\right\}
 \cup\left\{\dfrac12,\dfrac{2m+1}{4m+1}\right\}&m+4
\end{array}
}
\]
For each $r\in \{0,1,2,3\}$, each corresponding set in the second column is contained in $\F_n\cap I_m$, and hence contained in $\F_n\cap (\frac ab,\frac{a+1}{b-1})$. This gives the desired bound.

Suppose next that $b-2a=1$ and $a$ is even.  Then
\[
 \frac ab=\frac{2t}{4t+1}
 \quad\text{and}\quad
 \left(\frac ab,\frac{a+1}{b-1}\right)
 =J_t:=\left(\frac{2t}{4t+1},\frac{2t+1}{4t}\right)
\]
for some $t\ge1$, and $J_{t+1}\subseteq J_t$.  If $r=0$, then $t\le m-1$,
and $s/(2s-1)<(2m-1)/(4m-4)$ if and only if $s\ge m$.  If
$r\in\{1,2,3\}$, then $t\le m$, and
$s/(2s-1)<(2m+1)/(4m)$ if and only if $s\ge m+1$. The following table gives sufficient fractions in $\F_n\cap (\frac ab,\frac{a+1}{b-1})$ as required.
\[
{\renewcommand{\arraystretch}{1.15}
\begin{array}{c|l|c}
r&\text{fractions}&\text{cardinality}\\[1mm]\hline
0&\left\{\dfrac{s}{2s-1}:m\le s\le2m\right\}&m+1\\[4mm]
1,2&\left\{\dfrac{s}{2s-1}:m+1\le s\le2m+1\right\}
 \cup\left\{\dfrac12\right\}&m+2\\[4mm]
3&\left\{\dfrac{s}{2s-1}:m+1\le s\le2m+2\right\}
 \cup\left\{\dfrac12,\dfrac{2m+1}{4m+3}\right\}&m+4
\end{array}
}
\]

It remains to consider $b-2a=1$ with $a$ odd.  In this case
\[
 \frac ab=\frac{2t+1}{4t+3}
 \quad\text{and}\quad
 \left(\frac ab,\frac{a+1}{b-1}\right)
 =K_t:=\left(\frac{2t+1}{4t+3},\frac{t+1}{2t+1}\right)
\]
for some $t\ge0$, and $K_{t+1}\subseteq K_t$. 
Since the denominator $4t+3\leq n$, we have $t\leq m-1$ when $r\in \{0,1,2\}$; and by the additional
hypothesis, this also holds for $r=3$.
Since $s/(2s-1)<m/(2m-1)$ if and only if $s\ge m+1$, the following table provides the desired fractions.
\[
{\renewcommand{\arraystretch}{1.15}
\begin{array}{c|l|c}
r&\text{fractions}&\text{cardinality}\\[1mm]\hline
0&\left\{\dfrac{s}{2s-1}:m+1\le s\le2m\right\}
 \cup\left\{\dfrac12\right\}&m+1\\[4mm]
1,2&\left\{\dfrac{s}{2s-1}:m+1\le s\le2m+1\right\}
 \cup\left\{\dfrac12\right\}&m+2\\[4mm]
3&\left\{\dfrac{s}{2s-1}:m+1\le s\le2m+2\right\}
 \cup\left\{\dfrac{2m}{4m+1},\dfrac{2m+1}{4m+3}\right\}&m+4
\end{array}
}
\]
This completes the proof of Theorem~\ref{thm:central-cases}.
\end{proof}

\begin{proof}[\textbf{\textup{Proof of Theorem~\ref{thm:main}}}]
Let $n=4m+r$, where $0\le r\le3$, and assume that $m$ is sufficiently
large.  By \eqref{eq:f-definition}, choose $0\le i<j<|\F_n|$ such that
\[
 (a_j-a_i)(b_j-b_i)<0
 \quad\text{and}\quad
 j-i-1=f(n).
\]
The upper bounds in Theorem~\ref{thm:vandoorn-upper} imply
$f(n)\le n/4+4$.

Write $a/b=a_i/b_i$ and $c/d=a_j/b_j$.  Since $a/b<c/d$, the inequality
$(c-a)(d-b)<0$ implies $c\ge a+1$ and $d\le b-1$, and hence $a+1\le c\le d\le b-1$. In particular, we have $1\le a\le b-2$.  Moreover, 
\[
 \left|\F_n\cap\left(\frac ab,\frac{a+1}{b-1}\right)\right|\leq \left|\F_n\cap\left(\frac ab,\frac cd\right)\right|\le f(n).
\]

Suppose first that $b-2a\notin\{1,2\}$. Then Theorem~\ref{thm:general} gives
\[
 f(n)\ge\left|\F_n\cap
 \left(\frac ab,\frac{a+1}{b-1}\right)\right|
 \ge\frac{5n}{18}-O(\sqrt n\log n)
 >\frac n4+4
\]
for all sufficiently large $n$, a contradiction.  Hence
$b-2a\in\{1,2\}$.  By Theorem~\ref{thm:central-cases}, we have $f(n)\ge m+1$ for $r=0$, $f(n)\ge m+2$ for
$r\in\{1,2\}$, and $f(n)\ge m+4$ for $r=3$, except for the case that $r=3$ and
$a/b=(2m+1)/(4m+3)$.

It remains to assume $r=3$ and $a/b=(2m+1)/(4m+3)$. In this case, we have $c\ge2m+2$ and $d\le4m+2$, and they cannot both attain equality, since $\gcd(c,d)=1$.
We thus have
\begin{equation*}
 \frac cd\geq \min\{\frac{2m+2}{4m+1},\frac{2m+3}{4m+2}\}>\frac{2m+1}{4m}.
\end{equation*}

In this case, we have 
\begin{align*}
f(n)&\ge\left|\F_n\cap
 \left(\frac{2m+1}{4m+3},\frac cd\right)\right|\geq \left|\F_n\cap
 \left(\frac{2m+1}{4m+3},\frac{2m+1}{4m}\right]\right|\\
 &\geq \left|\{\frac{1}{2},\frac{2m+1}{4m}\}\cup\{\frac{s}{2s-1}:m+1\leq s\leq 2m+2\}\right|=m+4.
\end{align*}
This proves that the bounds in Theorem~\ref{thm:vandoorn-upper} are also lower bounds for $f(n)$, which completes the proof of Theorem~\ref{thm:main}.
\end{proof}

\section{Further remarks}\label{sec:5}

\subsection*{Proof sketch of Theorem~\ref{thm:all-n}}

For $4\le n\le5000$, direct enumeration of $\F_n$ gives the values stated
in Theorem~\ref{thm:all-n}.  Let $n\ge5000$.  The reduction in the proof of
Theorem~\ref{thm:main} already treats $b\le n/4$ and
$b-2a\in\{1,2\}$.  It therefore remains to consider
\[
 b>\frac n4,\qquad b-2a\notin\{1,2\}.
\]
For such $a/b$, let
\[
 C_n(a,b):=\left|\F_n\cap
 \left(\frac ab,\frac{a+1}{b-1}\right)\right|.
\]

For $n\ge5504798$, the main term in the proof of
Lemma~\ref{lem:determinant-count} is at least $5n/18$, while the total error
is at most
\[
 (9\lfloor\sqrt n\rfloor+3)
 \left(\frac{61}{100}\log(9\lfloor\sqrt n\rfloor+3)+\frac76\right).
\]
For every $n\ge5504798$, this quantity is smaller than $n/36-4$.  Hence
$C_n(a,b)>n/4+4\ge U(n)$.

It remains to consider $5000\le n\le5504797$.  For each $h\ne0$, the proof
of Lemma~\ref{lem:determinant-count} uses
\[
 |N_h|=
 \frac{\varphi(|h|)}{|h|}\operatorname{length}(I_h)
 +O\left(\sum_{d\mid |h|}1\right).
\]
We replace this $O$-term by the following upper bound for the actual error:
\[
 \sup_{\substack{J\text{ an interval}\\
                  \operatorname{length}(J)\le |h|}}
 \left|
 \left|\{j\in J\cap\mathbb Z:\gcd(j,|h|)=1\}\right|
 -\frac{\varphi(|h|)}{|h|}\operatorname{length}(J)
 \right|.
\]
By periodicity modulo $|h|$, the supremum is determined by one period.  For
a chosen Dirichlet approximation $p/q$, with $A=aq-bp$, the main term is
\[
 \frac bq\Hfun_{n/b}\left(A,\frac{b(A+p+q)}{b-1}\right),
\]
and the error is bounded by the sum of the preceding quantities over all
$h\ne0$ for which $I_h\ne\emptyset$.

We also use
\[
 G(x+1)-G(x)\ge\frac27\quad(x\ge1),\qquad
 G(x+1)-G(x)\ge\frac3{10}\quad(x\ge41),
\]
and
\[
 G(x+q)-G(x)\ge
 \max\left\{\frac{2q}{7},
 \frac{21273q-22546}{70000}\right\}
 \qquad(x\ge1).
\]

For $(p,q,A)\in\{(1,3,-1),(1,3,-3)\}$, the required bound is verified
directly by the program; in all other cases, we divide
$5000\le n\le 5504797$ into the following three ranges.
\begin{itemize}
\item For $5000\le n\le83387$, let
$R=\lfloor\sqrt{n/3}\rfloor$ and choose a Dirichlet approximation $p/q$
with the smallest possible denominator.  For every possible $p/q$, the
program uses the inequalities implied by this minimality and
\eqref{eq:H-formula} to verify that the main term minus the total error is
at least $U(n)$.

\item For $83388\le n\le181854$, let
$R=\lfloor\sqrt{n/8}\rfloor$.  For every $q\le R$, the main term is at least
\[
 \frac nq\max\left\{\frac{2q}{7},
 \frac{21273q-22546}{70000}\right\}.
\]
The program verifies that the total error is smaller than
\[
 \frac nq\max\left\{\frac{2q}{7},
 \frac{21273q-22546}{70000}\right\}-\frac n4-4.
\]

\item For $181855\le n\le5504797$, use the same value
$R=\lfloor\sqrt{n/8}\rfloor$.  The main term is at least $2n/7$.  The upper
bound for the total error is nondecreasing in $q$, so the program verifies
only at $q=R$ that the total error is smaller than
\[
 \frac n{28}-4.
\]
\end{itemize}

Together with the estimate for $n\ge5504798$, the enumeration for
$n\le5000$, and Theorem~\ref{thm:vandoorn-upper}, this proves
Theorem~\ref{thm:all-n}.

\subsection*{Code availability}
All code used to verify the numerical results in this paper, including~\eqref{eq:totient} and the proof of Theorem~\ref{thm:all-n}, is available in a public GitHub repository at \url{https://github.com/dct-cell/erdos/tree/main/1005}.

\subsection*{AI-use declaration}
During the preparation of this work, the authors used ChatGPT-5.6 to assist in generating candidate
proof strategies. All AI-generated
suggestions were verified and refined by the authors, who take full responsibility for the final content of the paper.

\end{document}